\documentclass[letterpaper,11pt]{article}
\usepackage[margin=1in]{geometry}
\usepackage[T1]{fontenc}

\usepackage[textsize=footnotesize, textwidth=2cm, color=yellow]{todonotes}

\usepackage{thmtools,cite}
\usepackage{graphicx}
\usepackage{amsmath,amsfonts,color}
\usepackage{amsthm,amssymb,mathtools}
\usepackage{xcolor}

\usepackage{centernot}
\usepackage{tikz}
\usetikzlibrary{positioning, shapes, backgrounds,decorations.pathmorphing}
\usepackage{fontawesome}
\usepackage{subcaption}

\definecolor{linkGreen}{RGB}{175,0,0}

\definecolor{linkblue}{RGB}{0, 102, 204} 

\usepackage[
unicode,
colorlinks=true,
citecolor=linkGreen,
linkcolor=linkblue,
urlcolor=linkblue,
filecolor=cyan!30!black
]{hyperref}

\usepackage[capitalize]{cleveref}

\newtheorem{theorem}{Theorem}

\newtheorem{lemma}[theorem]{Lemma}

\newtheorem{claim}{Claim}[theorem]

\newtheorem{observation}[theorem]{Observation}

\newcommand{\sic}{\xrightarrow{\cap}}

\DeclareMathOperator{\tw}{tw}
\newcommand{\cB}{\mathcal{B}}

\newcommand{\cX}{\mathcal{X}}
\newcommand{\cY}{\mathcal{Y}}
\newcommand{\cS}{\mathcal{S}}
\newcommand{\bN}{\mathbb{N}}

\newcommand{\Free}{\mathrm{Free}} 
\newcommand{\SubgraphFree}{\mathrm{Subgraph\text{-}Free}} 
\newcommand{\SiFree}{\mathrm{Si\text{-}Free}} 

\usepackage{enumerate,amssymb}
\usepackage{pifont}
\usepackage{xspace}
\newcommand{\mwis}{\textsc{Maximum Weight Independent Set}\xspace}
\newcommand{\mmaybewis}{\textsc{Maximum (Weight) Independent Set}\xspace}
\newcommand{\mis}{\textsc{Maximum Independent Set}\xspace}
\newcommand{\mim}{\textsc{Maximum Induced Matching}\xspace}

\newcommand{\sat}{\textsc{Satisfiability}\xspace}
\newcommand{\csat}{\textsc{Counting Satisfiability}\xspace}
\newcommand{\maybecsat}{\textsc{(Counting) Satisfiability}\xspace}

\usepackage{etoolbox}
\newtoggle{anonymous}
\togglefalse{anonymous}

\title{Graph Classes Closed under Self-intersection}

 \iftoggle{anonymous}{
	\author{ }
}{%
\author{Konrad K. Dabrowski \\ {\small Newcastle University, UK}
	\and Vadim Lozin  \\ {\small University of Warwick, UK}
        \and Martin Milani{\v c} \\ {\small University of Primorska, Slovenia}
        \and Andrea Munaro \\ {\small University of Parma, Italy}
        \and Dani\"el Paulusma  \\ {\small Durham University, UK}
	\and Viktor Zamaraev  \\ {\small University of Liverpool, UK}}
}

\date{}

\begin{document}

\maketitle
\begin{abstract}
A graph class is monotone if it is closed under taking subgraphs. It is well known that a monotone class defined by finitely many obstructions has bounded treewidth if and only if one of the obstructions is a so-called tripod, that is, a disjoint union of trees with exactly one vertex of degree~$3$ and paths. This dichotomy also characterizes exactly those monotone graph classes for which many \textup{\textsf{NP}}-hard algorithmic problems admit polynomial-time algorithms.  These algorithmic dichotomies, however, do not extend to the universe of all hereditary classes, which are classes closed under taking induced subgraphs. 
This leads to the natural question of whether we can extend known algorithmic dichotomies for monotone classes to larger families of hereditary classes.

In our paper we give an affirmative answer to this question by considering the family of hereditary graph classes that are closed under self-intersection. 
This family is known to be located strictly between the monotone and hereditary classes.
We show the following:
\begin{itemize}
\item We prove a new structural characterization of graphs in self-intersection-closed classes excluding a tripod.
In contrast to monotone classes excluding a tripod, these classes do not necessarily have bounded treewidth; in fact, they do not even need to be sparse.
\item
We use our characterization to give a complete dichotomy of \mis, and its weighted variant, for self-intersection-closed classes defined by finitely many obstructions: these problems are in {\sf P} if the class excludes a tripod and \textsf{NP}-hard otherwise. 
Our dichotomy generalizes several known results on \mis in the literature. 
\item
We also apply our characterization to obtain a dichotomy for \mim on self-intersection-closed classes of bipartite graphs defined by finitely many obstructions, and for \sat and \csat on  self-intersection-closed classes of (bipartite) incidence graphs defined by finitely many obstructions. 
\item We use our characterization to obtain a dichotomy for boundedness of clique-width for  self-intersection-closed classes of bipartite graphs defined by finitely many obstructions.
\end{itemize}
\end{abstract}

\newpage

\section{Introduction}
The celebrated project of Robertson and Seymour on graph minors produced a variety of fundamental results of both a structural and an algorithmic nature, and culminated in the proof of Wagner's conjecture stating that the minor relation on graphs is a well-quasi-order. 
This implies that every minor-closed class can be characterized by a finite set of minimal forbidden minors, and any family of classes in this universe, which is closed under taking subclasses, can be characterized by minimal classes not in the family. 
For instance, within the universe of minor-closed classes, the family of classes of bounded treewidth can be characterized by a unique ``obstruction'', i.e.~a unique minimal minor-closed class of unbounded treewidth, namely the class of planar graphs. 
What is interesting is that the family of classes of bounded treewidth coincides with the family of minor-closed classes where many \textup{\textsf{NP}}-hard algorithmic problems, such as \mis or \mim, admit polynomial-time algorithms (unless $\textsf{P} = \textup{\textsf{NP}}$).
The same is true for \sat represented by incidence graphs of CNF instances, i.e.~the problem is \textup{\textsf{NP}}-complete for planar incidence graphs \cite{MR1287974} and can be solved in polynomial time for incidence graphs of bounded treewidth \cite{DBLP:journals/cj/GottlobS08}. In other words, in the universe of minor-closed classes polynomial-time solvability of many \textup{\textsf{NP}}-hard algorithmic problems is equivalent to boundedness of treewidth.

These dichotomies, however, do not extend to the more general family of hereditary classes, i.e.~classes closed under taking induced subgraphs. 
The induced subgraph relation is not a well-quasi-order, it contains infinite antichains, e.g.~cycles. 
As a result, in this universe there exist classes characterized by infinitely many minimal forbidden induced subgraphs, and there exist families of classes with no minimal classes outside of the families. 
On the other hand, dichotomies 
are still possible for hereditary classes if we restrict ourselves to \textsl{finitely-defined} classes, i.e.~classes defined by finitely many forbidden induced subgraphs. 
In this case, we relax the notion of minimal classes to the notion of boundary classes, i.e.~minimal \textsl{limit} classes. 
Together, minimal and boundary classes are known as \textsl{critical classes}, or \textsl{critical properties}, and they enable a characterization of any family~$\cal F$ of finitely-defined hereditary classes that is closed under taking subclasses, in terms of critical properties: a class $X$ belongs to $\cal F$ if and only if it does not contain any critical class for $\cal F$. 
For instance, for the family of finitely-defined classes of bounded treewidth, there exist precisely four critical classes \cite{LR22}. Two of them, complete graphs and complete bipartite graphs, are minimal hereditary classes of unbounded treewidth, and the other two are boundary classes: the class $\cal S$ of forests in which every connected component is a tree with at most three leaves, and the class of line graphs of the graphs in $\cal S$. 

Every connected graph in $\cal S$ with three leaves is a subdivision of the claw $K_{1,3}$ and such graphs are known under various names, e.g.~subdivided claw, long claw and tripod. 
With slight abuse of terminology, we will refer to any graph in $\cal S$ as a \emph{tripod}. 
The class $\cal S$ is critical not only for graphs of bounded treewidth, but also for many algorithmic problems, including \mis \cite{MR2024261}, \mim \cite{MR2363374} and \sat \cite{MR3033644}. In other words, if a finitely-defined class $X$ contains $\cal S$, then \mis, \mim, and \sat are \textup{\textsf{NP}}-hard in $X$.
It was conjectured in \cite{MR3695266} that for the \mis problem, $\cal S$ is the only critical class, i.e.~by forbidding any graph from $\cal S$ as an induced subgraph we obtain a class where the problem can be solved in polynomial time. However, this has so far been confirmed only for a few very special tripods with small connected components (see, e.g.~\cite{GKPP22, BLM25}).
Recently, substantial progress has been made by proving weaker versions of the conjecture, such as in the setting of weakly sparse graph classes (i.e.~hereditary classes excluding both a clique and a biclique) \cite{ACPR24}, and by establishing quasi-polynomial-time solvability of \mis in classes that exclude a tripod \cite{MR4764849}.
Despite this progress, the conjecture remains out of reach to date. Under these circumstances, it is natural to restrict the universe to a family of classes between minor-closed and hereditary. An important family of this type is that of monotone classes.

A class of graphs is {\it monotone} if it is closed under taking subgraphs, not necessarily induced. 
The subgraph relation is much richer than the induced subgraph relation, but it still contains infinite antichains, which is bad news. 
The good news is that it contains a \textsl{canonical} infinite antichain, i.e.~an antichain that is unique in the sense that a monotone class is well-quasi-ordered under the subgraph relation if and only if it contains finitely many graphs from the canonical antichain~\cite{MR2493532}.
This antichain consists of the minimal forbidden (induced) subgraphs for the class  $\cal S$, which emphasizes the importance of $\cal S$ as a critical class. 

In the world of monotone classes, $\cal S$ is the unique critical class for many families, including classes of bounded treewidth and classes with polynomial-time solvable \mis, \mim and \sat problems (unless $\textsf{P} = \textup{\textsf{NP}}$); see also \cite{MR4867975} for many other algorithmic problems of this type. 
Therefore, we again find ourselves in the situation where polynomial-time solvability of the problems coincides with boundedness of treewidth, similarly to the case for minor-closed classes. 
This situation suggests that we step back and search for universes of classes that are strictly between monotone and hereditary classes. 
Is there life between these two worlds? 

To answer the last question, we observe that any monotone class that excludes a graph $G$ also excludes all graphs containing $G$ as a subgraph. 
But what if we exclude \textsl{only some} of these graphs?
One specific family of such classes was proposed in \cite{MR3812542} under the name \textsl{self-intersection-closed} classes. 
Informally, we say that a graph $H$ is a \emph{self-intersection} of a graph $G$ if $H$ is the intersection of a collection of isomorphic copies of $G$ (formal definitions will be given later). 
By definition, a self-intersection of $G$ is isomorphic to a subgraph of $G$, but in general not every subgraph of $G$ can be obtained in this way.
To distinguish between the two notions, we will refer to a self-intersection of a graph~$G$ as an \emph{si-subgraph} of $G$. 
A class of graphs is closed under self-intersection if for every graph $G$ the class contains all si-subgraphs of $G$. 
As we shall see later, the family of classes closed under self-intersection is located strictly between monotone and hereditary classes. 
Every class in this family can be characterized by a set of minimal forbidden si-subgraphs. 

\medskip
\noindent
{\bf Our Contribution.} 
In this paper, we focus on classes defined by forbidding graphs in $\mathcal{S}$ as si-subgraphs. 
We consider the following classical algorithmic problems as our testbed problems:
\begin{enumerate}
    \item \mwis: Given a graph $G = (V, E)$ and a weight function $w\colon V \to \mathbb{Q}_{\geq 0}$, the \mwis problem asks to find an independent set $I \subseteq V$ that maximizes the total weight $w(I) = \sum_{v \in I} w(v)$.
    The special case of the \mwis problem when the weight function $w$ is constantly equal to $1$ is the \mis problem.
    
    \item \mim: Given a graph $G = (V, E)$, an \emph{induced matching} is a set of edges $M \subseteq E$ such that no two edges in $M$ share an endpoint, and no edge in $E \setminus M$ connects two edges in $M$. The \mim problem asks to find an induced matching of maximum cardinality.

    \item \sat: Given a Boolean formula $\varphi$ in conjunctive normal form (CNF), the \sat problem asks whether there exists a truth assignment to the variables of $\varphi$ that makes the formula evaluate to true.

    \item \csat: Given a Boolean formula $\varphi$ in CNF, the \csat problem asks for the total number of satisfying truth assignments that make $\varphi$ evaluate to true.
\end{enumerate}

\noindent
We first focus on \mis and \mwis.
It was shown in \cite{MR3812542} that any self-intersection-closed class excluding a path~$P_k$ admits a polynomial-time algorithm for the \mis problem. 
Further progress towards the polynomial-time solvability of \mis in self-intersection-closed graphs excluding a connected tripod was obtained in~\cite{Sor23}, where a polynomial-time algorithm was developed under the additional assumption that the input graph has a vertex with co-degree bounded above by the logarithm of the number of vertices.

In the present paper, we extend these results to a {\it full} algorithmic dichotomy for the \mis and \mwis problems on classes defined by finitely many forbidden si-subgraphs. 
It is known~\cite{AK92} that a monotone class defined by finitely many forbidden subgraphs admits a polynomial-time algorithm for the \mis problem if and only if it excludes a graph from $\cS$ (unless $\textsf{P}=\textup{\textsf{NP}}$), and the same holds for the \mwis problem. 
We show that exactly the same dichotomies hold for \mis and \mwis on self-intersection-closed classes defined by finitely many forbidden si-graphs. Note, however, that in the case of monotone classes, excluding a tripod is equivalent to boundedness of treewidth~\cite{RS84}, whereas self-intersection-closed classes excluding a tripod can contain arbitrarily large complete graphs (see also \Cref{th:tSttt-main}).  

\begin{sloppypar}
The polynomial part of our dichotomy is based on a structural characterization of self-intersection-closed classes excluding a tripod, which is presented in Section~\ref{sec:no-long-claws}. 
In Section~\ref{sec:dichotomies}, we use this characterization to obtain the same dichotomy for \mim on classes of bipartite graphs defined by finitely many forbidden si-subgraphs, and for \sat and \csat on classes of (bipartite) incidence graphs defined by finitely many forbidden si-subgraphs, as 
described in Section~\ref{sec:dichotomies}.
In the same section, we also obtain a dichotomy for boundedness of clique-width for classes of bipartite graphs defined by finitely many forbidden si-subgraphs.   
\end{sloppypar}

\medskip
\noindent
{\bf Future Directions.}
Our paper not only extends some of the previously known results, but also opens a variety of new research directions. 
First, it is natural to explore other algorithmic problems for which the tripods constitute the only critical property in the universe of monotone classes and determine whether the structural characterization of self-intersection-closed classes excluding a tripod can be applied to develop efficient algorithms for such problems.

Second, we ask whether the dichotomy for classes of bipartite graphs of bounded clique-width defined by finitely many forbidden si-subgraphs can be extended to \textsl{all} classes defined by finitely many forbidden si-subgraphs.
We conjecture that this is not possible, i.e.~there exist classes excluding a tripod as an si-subgraph of unbounded clique-width. 
On the other hand, we believe that all classes excluding a tripod as an si-subgraph have bounded twin-width.

It would also be interesting to extend the dichotomy for \mim on classes of bipartite graphs defined by finitely many forbidden bipartite si-subgraphs to \textsl{all} classes defined by finitely many forbidden si-subgraphs, as well as to the weighted case of the problem.

One more important research direction is the complexity of \textup{\textsf{NP}}-hard algorithmic problems in classes defined by infinitely many forbidden si-subgraphs. 
For problems for which $\cal S$ is a critical class, polynomial-time solvability can only be obtained by forbidding a graph from $\cal S$ or by forbidding infinitely many graphs from the canonical antichain defining $\cS$ (unless $\textsf{P}=\textup{\textsf{NP}}$). 
This antichain consists of the cycles and the so-called $H$-graphs (see Figure~\ref{fig:Hk}). 
Therefore, one of the natural next steps is exploring the structure of self-intersection-closed classes excluding infinitely many cycles or $H$-graphs.
Finally, it would be interesting to discover other families of classes intermediate between monotone and hereditary classes.

\section{Preliminaries}

Let $G=(V,E)$ be a graph. 
We let $V(G)$ and $E(G)$ denote the vertex set and the edge set of $G$, respectively.
For a vertex $v \in V$, we let $N_G(v)$ denote the set of neighbours of $v$ in $G$, and let $\deg_G(v)$ denote the degree of $v$ in $G$, i.e.~the cardinality of $N_G(v)$. When $G$ is clear from the context, we simply write $N(v)$ and $\deg(v)$.
For a set $U \subseteq V$, we let $G[U]$ denote the subgraph of $G$ induced by $U$ and let $G-U$ denote the subgraph of $G$ induced by $V \setminus U$, i.e.~$G-U$ is obtained from $G$ by removing the vertices in $U$.
If $U$ consists of a single vertex $u$, we write $G - u$ instead of $G-U$ for brevity.
A set of pairwise adjacent (respectively, non-adjacent) vertices in $G$ is called a \emph{clique} (respectively, an \emph{independent set}) of $G$.
A (possibly empty) set $U \subseteq V$ is a \emph{cutset} in $G$ if the graph $G-U$ is disconnected and a \emph{clique cutset} if $U$ is both a clique and a cutset.
A \emph{cut-vertex} in a connected graph $G$ is a vertex $v$ such that $\{v\}$ is a cutset.
A graph is \emph{biconnected} if it is connected and has no cut-vertices.
A graph $H=(V,E')$ is a \emph{spanning subgraph} of $G$ if $E' \subseteq E$, and $H$ is a \emph{spanning supergraph} of $G$ if $E' \supseteq E$.

For two disjoint sets $A,B\subseteq V$, we say that $A$ and $B$ are \emph{complete} to each other (in $G$) if every vertex in $A$ is adjacent to every vertex in $B$.
Similarly, $A$ and $B$ are \emph{anticomplete} to each other if no vertex in $A$ is adjacent to a vertex in $B$.
The definitions are extended in the obvious way to the case when instead of one of the sets $A$ and $B$ we have just a vertex.

A set $U \subseteq V$ is a \emph{module} in $G$ if every vertex in $V \setminus U$ is either complete or anticomplete to $U$. 
Two distinct vertices $x,y \in V$ are \emph{twins} in $G$ if $N(x) \setminus \{x,y\} = N(y) \setminus \{x,y\}$; equivalently, $x,y$ are twins if $\{x,y\}$ is a module. Two non-adjacent twins are said to be \emph{false} twins.

A \emph{tree decomposition} of a graph $G$ is a pair $(T,\beta)$ where $T$ is a tree and $\beta$ is a mapping assigning to each node $t$ of $T$ a set $\beta(t)\subseteq V(G)$ such that, for each vertex $v\in V(G)$, the set $\{t\in V(T)\colon v\in \beta(t)\}$ induces a non-empty subtree of $T$, and for each edge $uv\in E(G)$, there exists a node $t$ of $T$ such that $\{u,v\}\subseteq \beta(t)$.
The \emph{width} of a tree decomposition $(T,\beta)$ is defined as the maximum value of $|\beta(t)|-1$ over all $t\in V(T)$.
The \emph{treewidth} of a graph $G$, denoted $\tw(G)$, is the smallest width of a tree decomposition of $G$.

We let $L(G)$ denote the \emph{line graph} of $G$, i.e.~the graph with vertex set $E(G)$, with distinct $e_1,e_2 \in E(G)$ being adjacent in $L(G)$ if and only if they share a vertex in $G$. 

We use the standard notation $K_{p,q}$, $K_s$, $P_s$ and $C_s$, for the complete bipartite graph with $p$ and $q$ vertices in the parts, the complete graph, the path and the cycle on $s$ vertices, respectively.
Given two vertex-disjoint graphs $G$ and $H$, we let $G+H$ denote the disjoint union of $G$ and $H$ and, for a positive integer $t$, let $tG$ denote the disjoint union of $t$ copies of $G$.

Given two graphs $G$ and $H$, we say that $G$ is \emph{$H$-free} if no induced subgraph of $G$ is isomorphic to~$H$.
For a set $M$ of graphs, we let $\Free(M)$ denote the hereditary class of graphs containing no graph from $M$ as an induced subgraph and let $\SubgraphFree(M)$ denote the monotone class of graphs containing no graph from $M$ as a subgraph. 
For finite $M$, we often omit the set notation, writing $\Free(F_1,\ldots, F_k)$ and $\SubgraphFree(F_1,\ldots, F_k)$ instead of $\Free(M)$ and $\SubgraphFree(M)$, respectively when $M = \{F_1,\ldots, F_k\}$.

For a positive integer $k$, we let $[k]$ denote the set $\{1,\ldots, k\}$.

\subsection{Self-intersection of Graphs and Classes Closed under Self-intersection}

Given two graphs $G_1=(V_1, E_1)$ and $G_2=(V_2,E_2)$, their \emph{intersection} is the graph $G_1 \cap G_2 = (V_1 \cap V_2, E_1 \cap E_2)$.
Given a graph $G=(V,E)$ and an injective mapping $\alpha$ defined on $V$, $G^{\alpha}$ denotes the graph $(V^{\alpha}, E^{\alpha})$, where $V^{\alpha} = \{ \alpha(v)\colon v \in V \}$ and $E^{\alpha} = \{ \{\alpha(v), \alpha(w)\}\colon \{v,w\} \in E \}$.

We will say that a graph $H$ is a \emph{self-intersection} of $G$, or that $G$ \emph{can be self-intersected} to $H$, denoted $G \sic H$, if there exist injective mappings $\alpha_1, \ldots, \alpha_k$ on $V$ such that $H = G^{\alpha_1} \cap \cdots \cap G^{\alpha_k}$. 
It is not hard to see that the relation $\sic$ is reflexive and transitive.

Clearly, if $H$ is a self-intersection of $G$, then $H$ is isomorphic to a subgraph of $G$.
However, in general, not every subgraph of $G$ can be obtained by self-intersecting $G$. 
For example, a complete graph can only be self-intersected to a complete subgraph.
To distinguish between the two notions, we refer to a self-intersection of $G$ as an {\it si-subgraph} of $G$. 
The following lemma was proved in~\cite{MR3812542}, where the notion of self-intersection was originally introduced.

\begin{lemma}[Alekseev and Sorochan~\cite{MR3812542}]\label[lemma]{lem:graph-sic-induced-subgraph}
 If $H$ is an induced subgraph of a graph $G$, then $G\xrightarrow{\cap} H$.
\end{lemma}

Below we prove a more general sufficient condition for a subgraph of $G$ to be an si-subgraph. 
Let a graph $H$ be isomorphic to a subgraph of a graph $G$ and let $\varphi\colon V(H) \rightarrow V(G)$ be a subgraph embedding of $H$ into $G$, that is, an injective mapping from $V(H)$ to $V(G)$ such that for every two adjacent vertices $u$ and $v$ of $H$, their images $\varphi(u)$ and $\varphi(v)$ are adjacent in $G$.
Then, the inclusion $\varphi(N_H(v))\subseteq N_{G}(\varphi(v)) \cap \varphi(V(H))$ holds for every $v \in V(H)$.

For a vertex $v \in V(H)$, the embedding $\varphi$ is \emph{induced with respect to $v$} if the neighbourhood of $\varphi(v)$ in $\varphi(V(H))$ in the graph $G$ is exactly $\varphi(N_H(v))$, i.e.~$\varphi(N_H(v)) = N_{G}(\varphi(v)) \cap \varphi(V(H))$. 

Note that if $\varphi$ is induced with respect to every vertex of $H$, then $\varphi$ is an \emph{induced} subgraph embedding of $H$ into $G$, that is, two distinct vertices $u$ and $v$ of $H$ are adjacent in $H$ if and only if their images $\varphi(u)$ and $\varphi(v)$ are adjacent in $G$.

\begin{lemma}\label[lemma]{lem:v-subgraph-sic}
    Let $G$ and $H$ be graphs. If for every vertex $v \in V(H)$ there exists a subgraph embedding of $H$ into $G$ that is induced with respect to $v$, then $G \sic H$.
\end{lemma}
\begin{proof}
    Let $[s]$ be the vertex set of $H$.
    For every vertex $v \in [s]$, define $G_v$  to be a graph isomorphic to $G$ such that the identity mapping $\varphi\colon [s] \rightarrow [s]$ is a subgraph embedding of $H$ into $G_v$ that is induced with respect to $v$, that is,
    \begin{enumerate}[(1)]
        \item $H$ is a subgraph of $G_v$, i.e.~$N_H(w) \subseteq N_{G_v}(w) \cap [s]$ holds for every $w \in [s]$; and
        
        \item $N_H(v) = N_{G_v}(v) \cap [s]$.
    \end{enumerate}
    
    Let $G' = \bigcap_{v \in [s]} G_v$. Then, from (1) and (2), we conclude that $N_H(v) = N_{G'}(v) \cap [s]$ holds for every $v \in [s]$, i.e.~$\varphi$ is a subgraph embedding of $H$ into $G'$ that is induced with respect to every vertex of $H$, implying that the subgraph of $G'$ induced by $[s]$ coincides with $H$. Thus, by \cref{lem:graph-sic-induced-subgraph}, we have $G \sic G' \sic H$, and therefore $G \sic H$ by transitivity.
\end{proof}

We say that a class $\mathcal{G}$ of graphs is {\it closed under self-intersection},  or \emph{self-intersection-closed}, if $G\in \mathcal{G}$ and $G \sic H$ imply $H\in \mathcal{G}$. 

\begin{sloppypar}
From the above discussion it follows that every monotone class is self-intersection-closed and every self-intersection-closed class is hereditary. 
Similarly to the forbidden subgraph characterization for monotone classes and the forbidden induced subgraph characterization for hereditary classes, every self-intersection-closed class admits a characterization in terms of forbidden si-subgraphs. 
For a set $M$ of graphs, we let $\SiFree(M)$ denote the class of graphs containing no graph in $M$ as an si-subgraph. 
Clearly, \[\SubgraphFree(M)\subseteq \SiFree(M)\subseteq \Free(M).\]
Both inclusions can be strict even in the case when $M$ consists of a single small graph in $\cal S$.
For example, if $M = \{P_1+P_3\}$, then the classes $\SubgraphFree(M)$, $\SiFree(M)$ and $\Free(M)$ are all pairwise distinct.
Indeed, it is not difficult to see that $\SubgraphFree(P_1+P_3)=\Free(C_3+\nobreak P_1,\allowbreak C_4,K_{1,3},\allowbreak P_1+\nobreak P_3,\allowbreak P_4,\textrm{paw},\textrm{diamond},K_4)$ and $\SiFree(P_1+P_3)=\Free(P_1+P_3,P_4,\textrm{paw},\textrm{diamond})$, where diamond is the $K_4$ minus an edge and paw is the graph formed by adding a pendant vertex to a~$K_3$.
In particular, note that if $F \in M$ is a graph on $r$ vertices and $M$ does not contain a complete graph, then $K_r \in \SiFree(M) \setminus \SubgraphFree(M)$.
\end{sloppypar}
\section{Excluding Tripods}\label{sec:no-long-claws}

In this section we study graph classes excluding tripods as si-subgraphs. We let $S_{t,t,t}$ denote the graph obtained from the claw $K_{1,3}$ by subdividing each of its three edges with $t-1$ new vertices, and let $tS_{t,t,t}$ denote the graph with~$t$ components, each isomorphic to $S_{t,t,t}$. Clearly, every tripod is an induced subgraph of $tS_{t,t,t}$ for some $t$, and hence without loss of generality we will consider classes excluding $tS_{t,t,t}$ as an si-subgraph. For $t \in \bN$, we let $\cX_t = \SiFree(tS_{t,t,t})$.
We analyse structural properties of graphs in $\cX_t$ that can be applied to develop efficient algorithms for the aforementioned algorithmic problems. 
Our main result in this section is the following.

\begin{restatable}{theorem}{tStttmain}\label{th:tSttt-main}
    Let $t \in \bN$. 
    Then, there exists $c \in \bN$ such that if
    a graph $G \in \cX_t$ contains neither false twins nor a clique cutset of size at most $2$, then $G$ is complete or $\tw(G) \leq c$.
\end{restatable}

To prove this result we will show that if a graph in $\cX_t$ without small clique cutsets or twins contains a large induced complete bipartite graph or a large complete subgraph, then this subgraph is connected to the rest of the graph in a very  structured way. 

\subsection{Handling Large Bicliques}\label{sec:large-bicliques}

In this section we show that if a graph $G$ from $\cX_t$ has no cut vertices, then any maximal induced complete bipartite graph in $G$ of sufficiently large size forms a module in $G$.

\begin{restatable}{theorem}{bicliqueToSttt}\label{th:biclique-to-Sttt}
    Let $t \in \mathbb{N}$ and let $G$ be a graph in $\cX_t$ with no cut vertices. 
    Then the vertex set of any maximal induced complete bipartite graph $K_{p,q}$ in $G$ with 
    $p,q \geq 3t^2+t+1$ is a module.
\end{restatable}

Before presenting the proof of this theorem at the end of the section, we first develop several technical ingredients, the main one being the following lemma.

\begin{lemma}\label[lemma]{lem:biclique-path}
    Let $t \in \bN$, and let $H$ be a graph whose vertex set is the union of three pairwise disjoint sets $A$, $B$ and $C$ such that:
    \begin{enumerate}
        \item $A$ and $B$ are independent sets each of size $3t^2+t$;
        \item $A$ and $B$ are complete to each other;
        \item $|C| \geq 2$ and $H[C]$ is a path with endpoints $x$ and $y$;
        \item both $x$ and $y$ have neighbours in $A \cup B$, and $C \setminus \{x,y\}$ is anticomplete to $A \cup B$;
        \item $\deg_H(x)=2$;
        \item $N(x) \cap (A \cup B) \neq N(y) \cap (A \cup B)$.
    \end{enumerate}
    Then $H \sic tS_{t,t,t}$.
\end{lemma}
\begin{proof}
    Let $n = |V(H)|$. By definition, $x$ has exactly two neighbours in $H$, one of which is in $A \cup B$ and the other of which is outside $A \cup B$. We will refer to $x$ as the \emph{anchor} vertex of $H$.

    Let $s = 3t^2+t$, and let
    $G$ be a graph on vertex set $[s]$ isomorphic to $tS_{t,t,t}$.
    For every vertex $v \in [s]$ of degree at most $2$ in $G$, define a graph $H_v$ with vertex set $[n]$ such that $H_v$ is isomorphic to $H$ and the identity mapping $\varphi\colon [s] \rightarrow [s]$ is a subgraph embedding of $G$ into $H_v$ that is induced with respect to $v$, i.e.
    \begin{enumerate}
        \item for every $i,j \in [s]$, if $i$ and $j$ are adjacent in $G$, then they are adjacent in $H_v$; and
        \item $N_G(v) = N_{H_v}(v) \cap [s]$.
    \end{enumerate}
    To see that $H_v$ exists for every vertex $v$ of $G$ of degree at most 2, it is enough to show the following claim.

    \begin{claim}\label{claim1}
    There exists a subgraph embedding of $G$ into $H$ that is induced with respect to $v$.\end{claim}
    
    \begin{proof}[Proof of \Cref{claim1}]
    Denote by $R$ the connected component of $G$ that contains $v$. 
    Let $w$ be the neighbour of $v$ in $G$ on the path from $v$ to the vertex of degree $3$ in $R$, and let $P$ be the shortest path in $R$ from $v$ to a vertex of degree $1$ that does not contain $w$ (such a path consists only of $v$ if $\deg_G(v)=1$).
    We first embed $R$ into $H$ as a subgraph. We do this in such a way that: 
    \begin{enumerate}
        \item The vertex $v$ is mapped to the anchor $x$ and the vertex $w$ is mapped to the unique neighbour of $x$ in $A \cup B$;
        \item the path $P$ is mapped to the path $H[C]$ if $H[C]$ is at least as long as $P$, and otherwise the initial segment of $P$ is mapped to $H[C]$ and the remaining segment of $P$ is mapped to $(A \cup B) \setminus \{w\}$ arbitrarily, making sure that the edge relation is preserved;
        \item the remaining part of $R$, i.e. $R-(V(P)\cup\{w\})$, is mapped to $A \cup B$ arbitrarily by extending the embedding of $P$ and $w$ to a subgraph embedding of $R$.
    \end{enumerate}
    The remaining $t-1$ connected components of $G$ are embedded into a part of $H[A \cup B]$ vertex-disjoint from the image of $R$. 
    The lower bound on the sizes of $A$ and $B$ guarantees that there are enough vertices to do so.
    We note that the above subgraph embedding is induced with respect to $v$. 
    Indeed, if $\deg_G(v)=2$, then this is the case because $v$ is mapped to the anchor vertex $x$ of $H$ and $\deg_H(x) = 2$. 
    If $\deg_G(v)=1$, then this is the case because $v$ is mapped to the anchor vertex $x$, and no vertex of $G$ is mapped to the neighbour of $x$ outside $A \cup B$. 
    This completes the proof of the claim.
    \end{proof}

    Let $F = \bigcap_{v} H_v$, where the intersection is over all vertices $v \in [s]$ with $\deg_G(v) \leq 2$. Note that $G$ is a subgraph of $F$, as it is a subgraph of each $H_v$. Furthermore, the identity mapping $\varphi\colon [s] \rightarrow [s]$ is a subgraph embedding of $G$ into $F$ that is induced with respect to every vertex $v \in [s]$ with $\deg_G(v) \leq 2$.
    This implies that the embedding $\varphi$ of $G$ into $F$ is almost induced, in the sense that each connected component $S_{t,t,t}$ of $G$ is embedded  by $\varphi$ as an induced subgraph due to \cref{lem:v-subgraph-sic} (because $S_{t,t,t}$ has only one vertex of degree more than~$2$), and between images of different connected components there could be edges only between the images of vertices of degree $3$.
    
    Our goal is to remove these possible edges between the images of different copies of $S_{t,t,t}$ by intersecting $F$ with some more graphs isomorphic to $H$. For every $i \in \{1,\ldots,t\}$, let $Q_i,R_i \subseteq [s]$ denote the sets of vertices such that the graph $F[Q_i \cup R_i]$ is isomorphic to $S_{t,t,t}$, and $Q_i$ and $R_i$ are the parts of the bipartition of $S_{t,t,t}$, where $Q_i$ contains the vertex of degree $3$. 
    Let $F'$ be the graph on the vertex set $[s]$ that is the disjoint union of $t$ complete bipartite graphs, where the $i$-th complete bipartite graph has parts $Q_i$, $R_i$. Note that by intersecting $F$ with $F'$ we will achieve our goal of removing edges between different copies of $S_{t,t,t}$ in $F$ while preserving all other edges, i.e. $F \cap F'$ is isomorphic to $tS_{t,t,t}$.
    Thus, to complete the proof, it suffices to show that $H \sic F'$.

    To do so, for convenience, we first self-intersect $H$ to obtain its subgraph $H'$ induced by the set $A \cup B$, i.e. $H'$ is a complete bipartite graph with parts $A,B$ of equal size $3t^2+t$. We now show that $H' \sic F'$.
    Suppose that the vertices of $H'$ are labelled in such a way that
    \[
        \bigcup_{i=1}^{t} Q_i \subseteq A 
        \quad \text{ and } \quad
        \bigcup_{i=1}^{t} R_i \subseteq B,
    \]
    which is possible as $A$ and $B$ are of sufficiently large size.
    
    For every $j \in [t]$, define $Z_j$ to be the complete bipartite graph isomorphic to $H'$ and obtained from $H'$ by swapping the sides of $Q_j$ and $R_j$ and balancing the parts using the vertices outside $Q_i,R_i$, $i \in [t]$. More formally, 
    one part of $Z_j$ is obtained from $A$ by removing $Q_j$ and adding $R_j$, and the other part is obtained from $B$ by removing $R_j$ and adding $Q_j$. 
    Note that by swapping $Q_j$ and $R_j$, the difference between the sizes of the two parts becomes $2$ or $4$, depending on the parity of $t$, so we move one or two vertices, respectively, outside $Q_i,R_i$, $i \in [t]$ from the larger part to the smaller part to make them equal in size.

    Observe that for every $i \in [t]$, the sets $R_i$ and $Q_i$ are complete to each other in $H'$ and in all the graphs $Z_j$, $j \in [t]$, so $F'$ is a subgraph of each of these graphs. Furthermore, for any $i, j \in [t]$ with $i \neq j$, each pair of sets $(Q_i, R_j)$, $(Q_j, R_i)$, $(Q_i,Q_j)$, and $(R_i,R_j)$ is anticomplete in $H'$ or $Z_i$. 
Thus, the intersection $\left(\bigcap_{j \in [t]} Z_j\right) \cap H'$ contains $F'$ as an induced subgraph, and therefore, by \cref{lem:graph-sic-induced-subgraph}, we have $H' \sic F'$.
\end{proof}

\begin{lemma}\label[lemma]{lem:biclique+v-subgraph}
    Let $H$ be a graph with vertex set $A \cup B \cup \{v\}$ such that:
    \begin{enumerate}
        \item $A$, $B$, $\{v\}$ are pairwise disjoint;
        \item $A$ is an independent set or a clique;
        \item $A$ and $B$ are complete to each other;
        \item $v$ has a neighbour in $A$ and a non-neighbour in $A$.
    \end{enumerate}
    Let $F$ be a spanning subgraph of $H$ obtained by removing some edges between $v$ and vertices in $A$.
    Then $H \sic F$.
\end{lemma}
\begin{proof}
    Let $|A| = k$ and
    suppose without loss of generality that $A = [k]$, $N_H(v) \cap A = [s]$, and $N_{F}(v) \cap A = [t]$, for some $0 \leq t < s < k$.

    For every vertex $x \in [s] \setminus [t]$, i.e.~every vertex in $A$ that is adjacent to $v$ in $H$ but non-adjacent to $v$ in $F$, let $Z_x$ be the graph obtained from $H$ by swapping the adjacencies between $v,x$ and $v,k$. More precisely, $E(Z_x) = (E(H) \setminus \{v,x\}) \cup \{v,k\}$. Note that $x$ and $k$ have the same neighbourhoods in $H-v$, and $v$ is adjacent to exactly one of them in each $H$ and $Z_x$, which implies that $Z_x$ is isomorphic to $H$.

    We claim that the intersection $(\bigcap_{x \in [s] \setminus [t]} Z_x) \cap H$ coincides with $F$. 
    Indeed, for each $x \in [s] \setminus [t]$, the graph $Z_x-v$ is equal to $H-v$. 
    Furthermore, $v$ is adjacent to $[t] \cup (N_H(v)\cap B)$
    in each of the graphs $H$, $Z_x$, $x \in [s] \setminus [t]$, and for every $y \in [k] \setminus [t]$, the vertex $v$ is not adjacent to $y$ in at least one of these graphs.
    Thus, in the intersection $(\bigcap_{x \in [s] \setminus [t]} Z_x) \cap H$, the neighbourhood of $v$ is equal to $N_F(v) = [t] \cup (N_H(v)\cap B)$.    
    Consequently, $F = (\bigcap_{x \in [s] \setminus [t]} Z_x) \cap H$, and therefore $H \sic F$.
\end{proof}

In this section, we use \cref{lem:biclique+v-subgraph} only in the case when $A$ is an independent set; the case when $A$ is a clique will be used in \Cref{sec:large-cliques}.

Let $X_{p}$ denote the graph obtained from $K_{p,p}$ by adding a vertex that has exactly two neighbours, such that these neighbours belong to the same part of the $K_{p,p}$. Similarly, let $Y_{p}$ denote the graph obtained from $K_{p,p}$ by adding a vertex that has exactly two neighbours, such that these neighbours belong to different parts of the $K_{p,p}$. See \cref{fig:main} for an illustration.

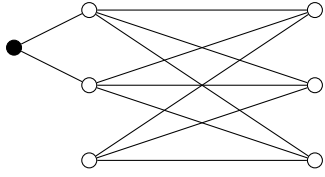
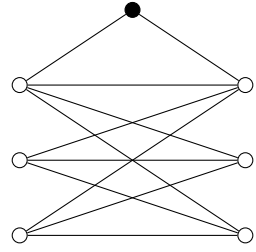
\begin{figure}[ht]
\centering
\begin{subfigure}[t]{0.45\textwidth}
\centering
\begin{tikzpicture}[every node/.style={circle, draw, fill=white}, node distance=1cm, inner sep=2pt]
  \node (a1) at (0,2) {};
  \node (a2) at (0,1) {};
  \node (a3) at (0,0) {};
  \node (b1) at (3,2) {};
  \node (b2) at (3,1) {};
  \node (b3) at (3,0) {};
  \node[fill=black] (x) at (-1,1.5) {};

  \foreach \i in {a1,a2,a3}
    \foreach \j in {b1,b2,b3}
      \draw (\i) -- (\j);

  \draw (x) -- (a1);
  \draw (x) -- (a2);
\end{tikzpicture}
\caption{\(X_3\): a vertex (black) is adjacent to two vertices in the same part of $K_{3,3}$.}
\label{fig:X3}
\end{subfigure}
\hfill
\begin{subfigure}[t]{0.45\textwidth}
\centering
\begin{tikzpicture}[every node/.style={circle, draw, fill=white}, node distance=1cm, inner sep=2pt]
  \node (a1) at (0,2) {};
  \node (a2) at (0,1) {};
  \node (a3) at (0,0) {};
  \node (b1) at (3,2) {};
  \node (b2) at (3,1) {};
  \node (b3) at (3,0) {};
  \node[fill=black] (y) at (1.5,3) {};

  \foreach \i in {a1,a2,a3}
    \foreach \j in {b1,b2,b3}
      \draw (\i) -- (\j);

  \draw (y) -- (a1);
  \draw (y) -- (b1);
\end{tikzpicture}
\caption{\(Y_3\): a vertex (black) is adjacent to one vertex in each part of $K_{3,3}$.}
\label{fig:Y3}
\end{subfigure}

\caption{The graphs $X_3$ and $Y_3$.}
\label{fig:main}
\end{figure}

\begin{lemma}\label[lemma]{lem:XpYp-tSttt}
    Let $t,p \in \bN$, where $p \geq 3t^2+t+1$. 
    Let $H'$ be a graph isomorphic to $X_p$ or $Y_p$ with vertex set $A' \cup B' \cup \{x\}$, where $A'$ and $B'$ are complete to each other and $x$ has degree $2$. Then $H' \sic tS_{t,t,t}$.
\end{lemma}
\begin{proof}
    The lemma follows from \cref{lem:biclique-path} by noticing that both the graphs $X_p$ and $Y_p$ contain an induced subgraph satisfying the conditions of \cref{lem:biclique-path}.
    Indeed, let $y,z$ be the neighbours of $x$ in $A' \cup B'$. Then, by taking $C=\{x,y\}$, $A \subseteq A'$ and $B \subseteq B'$ such that $|A| = |B| = 3t^2 + t$, $y \not\in A \cup B$ and $z \in A \cup B$, one can see that $H = H'[A \cup B \cup C]$ satisfies the conditions of \cref{lem:biclique-path}. Thus, $H' \sic H \sic tS_{t,t,t}$.
\end{proof}

\begin{lemma}\label[lemma]{lem:biclique-1}
    Let $p \geq 3$ and
    let $H$ be a graph with vertex set $A \cup B \cup \{v\}$ such that:
    \begin{enumerate}
        \item $A$, $B$, $\{v\}$ are pairwise disjoint;
        \item $A$ and $B$ are independent sets that are complete to each other;
        \item $|A| \geq p$ and $|B| \geq p$;
        \item $v$ has at least two neighbours in $A \cup B$;
        \item $v$ has a neighbour and a non-neighbour in at least one of the sets $A,B$.
    \end{enumerate}
    Then $H \sic X_p$ or $H \sic Y_p$.
    In particular, if $p \geq 3t^2+t+1$ for some $t \in \bN$, then $H \sic tS_{t,t,t}$.
\end{lemma}
\begin{proof}
    By removing some vertices of the graph $H$ from $A$ and/or $B$ we can make sure that $|A| = |B| = p$ while the resulting graph still satisfies the assumptions of the lemma. Thus, using \cref{lem:graph-sic-induced-subgraph}, we may assume without loss of generality that $|A| = |B| = p$.

    If $v$ has a neighbour and a non-neighbour in both $A$ and $B$, then by applying \cref{lem:biclique+v-subgraph} twice, we can conclude that $H \sic Y_p$.

    If $v$ is complete to one of the sets $A$ and $B$, say to $B$, then by assumption, $v$ has a neighbour and a non-neighbour in $A$. Let $k$ be the number of neighbours of $v$ in $A$, and define $H'$ to be the graph with vertex set $A \cup B \cup \{v\}$ such that $H'-v = H-v$, $v$ is complete to $A$ and $v$ has exactly $k$ neighbours in $B$. Since $|A|=|B|$, it is easy to see that $H'$ is isomorphic to $H$.
    Furthermore, note that in the intersection $H \cap H'$, the vertex $v$ has a neighbour and a non-neighbour in both $A$ and $B$.
    Thus, by the previous case, $H \sic H \cap H' \sic Y_p$.

    If $v$ is anticomplete to one of the sets, say $B$, then by assumption, $v$ has at least two neighbours and at least one non-neighbour in $A$. Thus, by \cref{lem:biclique+v-subgraph}, one concludes that $H \sic X_p$.

    The above establishes that $H \sic X_p$ or $H \sic Y_p$. 
    Now, if $p \geq 3t^2+t+1$ for some $t \in \bN$, then, by \cref{lem:XpYp-tSttt}, we have $H \sic tS_{t,t,t}$.
\end{proof}

We are now ready to prove the main result of this section, which we restate for convenience.

\bicliqueToSttt*
\begin{proof}
    Let $s = 3t^2+t+1$, and
    let $A,B \subseteq V(G)$ be two disjoint independent sets in $G$ such that $G[A \cup B]$ is a maximal induced complete bipartite graph in $G$ and $|A| \geq s$ and $|B| \geq s$. 

    To prove the statement we need to show that every vertex $x$ outside $A \cup B$ is either complete or anticomplete to $A \cup B$.
    Suppose there exists a vertex $x \in V(G) \setminus (A \cup B)$ that is neither complete nor anticomplete to $A \cup B$. We will show that this implies $G \sic tS_{t,t,t}$. Consider two cases.

    \textbf{Case 1:} \textit{$x$ has at least two neighbours in $A \cup B$}. 
    Note, due to the maximality of $G[A \cup B]$, the vertex $x$ cannot be complete to one of the sets $A$ and $B$, and anticomplete to the other.
    Furthermore, since $x$ is not complete to $A \cup B$, it must have a neighbour and a non-neighbour in at least one of the parts $A$ and $B$. 
    Then $G \sic tS_{t,t,t}$ by \cref{lem:biclique-1}.

    \textbf{Case 2:} \textit{$x$ has exactly one neighbour in $A \cup B$, which we denote by $y$.} Since $G$ has no cut vertices, $G-y$ is connected. Let $P$ be a shortest path from $x$ to $(A \cup B) \setminus \{y\}$ in $G-y$. Then the graph $H=G[A \cup B \cup V(P)]$ satisfies the conditions of \cref{lem:biclique-path} and therefore $G \sic H \sic tS_{t,t,t}$.
\end{proof}

\subsection{Handling Large Cliques}\label{sec:large-cliques}

In this section we show that if a connected graph $G$ from $\cX_t$ has no clique cutsets of size $1$ or $2$, then $G$ is either a complete graph or contains no large cliques.

\begin{restatable}{theorem}{cliqueToSttt}\label{th:clique-to-Sttt}
    Let $t \in \mathbb{N}$.
    Every graph in $\cX_t$ with no clique cutsets of size at most $2$ is complete or $K_r$-free, where $r = 3t^2+t+2$.   
\end{restatable}

Before proving this theorem, we first establish some auxiliary technical results that will be used in the argument.

\begin{lemma}\label[lemma]{lem:clique-3-path}
    Let $t \in \bN$ and let $H$ be a graph with vertex set $A \cup T$ such that:
    \begin{enumerate}
        \item $|A| \geq 3t^2 + t + 1$;
        
        \item $A$ is a clique in $H$;

        \item $T = V(Q_1) \cup V(Q_2) \cup V(Q_3)$, where $Q_1$, $Q_2$, $Q_3$ are paths in $H$ from a vertex $w \not\in A$ to $A$ such that the subpaths $Q_1 - w$, $Q_2 - w$, $Q_3 - w$ are pairwise vertex disjoint;

        \item $\deg_H(w) = 3$;
        
        \item $|A \cap T| = 3$ and the three vertices in the intersection are the endpoints of $Q_1$, $Q_2$, $Q_3$ that are different from $w$; we denote these vertices by $y_1$, $y_2$, $y_3$, respectively;

        \item $H[V(Q_1) \cup V(Q_2)]$ is a chordless cycle.
        
        \item $A \setminus T$ and $T \setminus A$ are anticomplete to each other.
    \end{enumerate}
    Then $H \sic tS_{t,t,t}$.
\end{lemma}
\begin{proof} 
   For convenience, \cref{fig:clique+tripod} provides a schematic illustration of the structure of the graph $H$. 

\begin{figure}
\centering

\begin{tikzpicture}[scale=0.8]

\tikzset{vertex/.style={circle, draw, fill=white, inner sep=2pt}}

\draw[fill=black!2] (0,0) circle (2.3);
\node at (-2.7,0) {\Large $A$};

\node[label=left:$Q_1$] at (-1.9,1.8) {};
\node[label=right:$Q_2$] at (-0.1,2.8) {};
\node[label=right:$Q_3$] at (1.9,1.8) {};

\node[vertex, label=below:$y_1$] (y1) at (-1.75,0.5) {};
\node[vertex, label=below:$y_2$] (y2) at (0,1.85) {};
\node[vertex, label=below:$y_3$] (y3) at (1.75,0.5) {};

\node[circle, draw, fill=black!10, inner sep=2pt, label=above:$w$] (w) at (0,5) {};

\coordinate (q1a) at (-1.5,3.8);
\coordinate (q1b) at (-2,3);
\draw[decorate, decoration={snake, amplitude=0.3mm}, very thick, blue] (w) -- (q1a) -- (q1b) -- (y1);

\coordinate (q2a) at (0,4.3);
\coordinate (q2b) at (0,3);
\draw[decorate, decoration={snake, amplitude=0.3mm}, very thick, blue] (w) -- (q2a) -- (q2b) -- (y2);

\coordinate (q3a) at (1.4,4.1);
\coordinate (q3b) at (2,3.2);
\draw[decorate, decoration={snake, amplitude=0.3mm}] (w) -- (q3a) -- (q3b) -- (y3);

\node[circle,inner sep=1.5pt, fill=blue] (q1_1) at (q1a) {};
\node[circle,inner sep=1.5pt, fill=blue] (q1_2) at (q1b) {};

\node[circle,inner sep=1.5pt, fill=blue] (q2_1) at (q2a) {};
\node (q2_2) at (q2b) {};

\node[circle,inner sep=1.5pt, fill=black] (q3_1) at (q3a) {};
\node[circle,inner sep=1.5pt, fill=black]  (q3_2) at (q3b) {};

\draw[gray] (q3_1) -- (q1_1);
\draw[gray] (q3_2) -- (q2_1);
\draw[gray] (q3_1) -- (q1_2);

\end{tikzpicture}

\caption{An illustration of the structure of the graph $H$ from \cref{lem:clique-3-path}. There are no edges between the vertices of $Q_1$ and $Q_2$, except for the edge $(y_1, y_2)$. This is indicated by highlighting these paths in blue and using thicker lines. In contrast, edges may exist between the vertices of $Q_3$ and those of $Q_1$ and $Q_2$.}\label{fig:clique+tripod}
\end{figure}
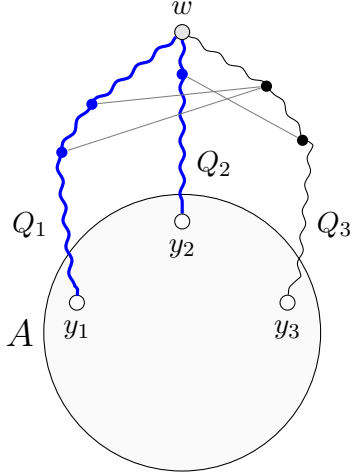

    Let $G$ be a graph isomorphic to $tS_{t,t,t}$.
    Due to \cref{lem:v-subgraph-sic}, to prove the statement it is enough to show that, for every vertex $v$ of $G$, there exists a subgraph embedding of $G$ into $H$ that is induced with respect to~$v$. 
    We do this in the rest of the proof.
    
    Let $v$ be a vertex of $G$. Let $R$ denote the connected component of $G$ that contains $v$, and consider three cases depending on the degree of $v$ in $G$.

    \textbf{Case 1: $\deg_G(v) = 3$.} In this case, we embed $R$ by mapping $v$ to the vertex $w$ of $H$, and the three paths branching from $v$, respectively, to the paths $Q_1, Q_2, Q_3$ and possibly to some vertices in $A$, if some of the branching paths in $R$ are longer than their respective paths $Q_1$, $Q_2$ and $Q_3$. Once $R$ is embedded, $A$ still contains at least $(3t^2 + t + 1) - (3t+1) > (t-1)(3t^2 + 1)$ available vertices to which the remaining $t-1$ connected components of $G$ can be mapped arbitrarily, as $A$ is a clique. Since the degree of $v$ in $G$ and the degree of its image in $H$ are the same, the embedding is trivially induced with respect to $v$. 

    \textbf{Case 2: $\deg_G(v) = 2$.} Let $H'$ be the induced subgraph of $H$ obtained by removing all vertices of $Q_3$ except $w$ from $H$.
    We will establish a subgraph embedding of $G$ into $H'$ that is induced with respect to $v$. Obviously, such an embedding is also a subgraph embedding of $G$ into $H$ and it is induced with respect to $v$.
    Let $A' = A \cap V(H')$ and note that $|A'| = |A|-1$.
    Recall that $y_1$ is the end vertex of $Q_1$ that belongs to $A'$. Let $u$ be the neighbour of $y_1$ in $Q_1$. Assumptions 5, 6 and 7 of the lemma imply that $u$ has degree $2$ in $H'$.
    
    Let $x$ be the neighbour of $v$ in $G$ on the path from $v$ to the vertex of degree $3$ in $R$, and let $P$ be the shortest path in $R$ from $v$ to a vertex of degree $1$ that does not contain $x$. 
    We embed $G$ into $H'$ as a subgraph by mapping $v$ to $u$, $x$ to $y_1$, and $P$ to the union of the paths $Q_1$ and $Q_2$ and possibly some vertices in $A'$ if this union is shorter than $P$. 
    The remaining part of $R$, i.e. $R-(V(P)\cup \{x\})$, is mapped to the available vertices in $A'$ arbitrarily.
    After embedding $R$, there are still at least $|A'| - 3t = |A|-1 -3t \geq 3t^2+t - 3t > (t-1)(3t+1)$ vertices available in $A'$, so the remaining $t-1$ connected components of $G$ can be embedded into  $A'$ vertex-disjointly from the image of $R$. 
    
    As in the previous case, the degree of $v$ in $G$ is equal to the degree of its image in $H'$, and thus the embedding is induced with respect to $v$.  

    \textbf{Case 3: $\deg_G(v) = 1$.}
    Let $u$ be the neighbour of $y_1$ in $Q_1$, and
    let $H'$ be the induced subgraph of $H$ obtained by removing from $H$ all vertices of $Q_1$, $Q_2$ and $Q_3$ except $u$ and $y_1$. Assumptions 3, 5 and 7 of the lemma imply that the degree of $u$ in $H'$ is $1$.
    As in the previous case, it is enough to establish the desired subgraph embedding of $G$ into $H'$.
    
    Let $A' = A \cap V(H')$, and note that $|A'| = |A|-2 \geq 3t^2+t-1$.
    We embed $G$ into $H'$, by mapping $v$ to $u$, the unique neighbour $x$ of $v$ in $G$ to $y_1$, and $V(G) \setminus \{v,x\}$ to $A' \setminus \{y_1\}$ arbitrarily.
    The latter is possible since $|A' \setminus \{y_1\}| \geq 3t^2+t-2 = |V(G) \setminus \{v,x\}|$.
    Since the degree of $v$ in $G$ is equal to the degree of its image in $H'$, the embedding is induced with respect to $v$.
\end{proof}

The following lemma is a corollary of \cref{lem:biclique+v-subgraph} in the case when $A$ is a clique and $B = \emptyset$.

\begin{lemma}\label[lemma]{lem:clique+v-subgraph}
    Let $H$ be a graph with vertex set $A \cup \{v\}$ such that:
    \begin{enumerate}
        \item $A$ and $\{v\}$ are disjoint;
        \item $A$ is a clique;
        \item $v$ has a neighbour and a non-neighbour in $A$.
    \end{enumerate}
    Let $F$ be a spanning subgraph of $H$ obtained by removing some edges incident to $v$.
    Then $H \sic F$.
\end{lemma}

We are now ready to prove the main result of this section, which we restate for convenience.

\cliqueToSttt*
\begin{proof}
\begin{figure}
\centering
\begin{subfigure}[t]{0.3\textwidth}
\centering
\begin{tikzpicture}[scale=0.8]

\tikzset{vertex/.style={circle, draw, fill=white, inner sep=2pt}}

\draw[fill=black!2] (0,0) circle (2.3);
\node at (-2.8,0) {\Large $A'$};

\node[vertex, label=below:$y_1$] (y1) at (-1.3,1.6) {};
\node[vertex, label=below:$y_2$] (y2) at (0,1.85) {};
\node[vertex, label=below:$y_3$] (y3) at (1.3,1.6) {};

\node[circle, draw, fill=black!10, inner sep=2pt, label=above:$w$] (w) at (0,4) {};

\draw[very thick,blue] (w) -- (y1);
\draw[very thick,blue] (w) -- (y2);
\draw (w) -- (y3);

\end{tikzpicture}
\caption{Case 1: $w$ has at least three neighbours in $A'$.}
\label{fig:Th3-case1}
\end{subfigure}
\hfill
\begin{subfigure}[t]{0.3\textwidth}
\centering
\begin{tikzpicture}[scale=0.8]

\tikzset{vertex/.style={circle, draw, fill=white, inner sep=2pt}}

\draw[fill=black!2] (0,0) circle (2.3);
\node at (-2.8,0) {\Large $A'$};

\node[vertex, label=below:$y_1$] (y1) at (-1.3,1.6) {};
\node[vertex, label=below:$y_2$] (y2) at (0,1.85) {};
\node[vertex, label=below:$y_3$] (y3) at (1.75,1) {};
\node[vertex, label=right:$z$] (z) at (2.1,1.7) {};

\node[circle, draw, fill=black!10, inner sep=2pt, label=above:$w$] (w) at (-0.5,4) {};

\draw[very thick,blue] (w) -- (y1);
\draw[very thick,blue] (w) -- (y2);
\draw (y3) -- (z);

\coordinate (q3a) at (1.4,3.8);
\coordinate (q3b) at (2,3);
\draw[decorate, decoration={snake, amplitude=0.3mm}] (w) -- (q3a) -- (q3b) -- (z);
\node at (2.3,3.3) {$P$};

\end{tikzpicture}
\caption{Case 2: $w$ has exactly two neighbours in $A'$.}
\label{fig:Th3-case2}
\end{subfigure}
\hfill
\begin{subfigure}[t]{0.3\textwidth}
\centering
\begin{tikzpicture}[scale=0.8]

\tikzset{vertex/.style={circle, draw, fill=white, inner sep=2pt}}

\draw[fill=black!2] (0,0) circle (2.3);
\node at (-2.8,0) {\Large $A'$};


\node[vertex, label=below:$y_1$] (y1) at (-1.75,0.5) {};
\node[vertex, label=left:$u$] (u) at (-2.4,1.3) {};
\node[vertex, label=below:$y_2$] (y2) at (0.5,1.85) {};
\node[vertex, label=below:$y_3$] (y3) at (1.75,1) {};

\node[circle, draw, fill=black!10, inner sep=2pt, label=above:$w$] (w) at (-1.5,3.3) {};

\coordinate (q1a) at (-2.3,2.3);
\draw[decorate, decoration={snake, amplitude=0.3mm}, very thick, blue] (w) -- (q1a) -- (u);
\draw[very thick, blue] (y1) -- (u);
\node[label=right:$P'$] at (-1.9,2.8) {};

\coordinate (q2a) at (-0.2,3.3);
\draw[decorate, decoration={snake, amplitude=0.3mm}, very thick, blue] (w) -- (q2a) -- (y2);

\coordinate (q3a) at (0.4,3.8);
\coordinate (q3b) at (1.5,3.2);
\draw[decorate, decoration={snake, amplitude=0.3mm}] (w) -- (q3a) -- (q3b) -- (y3);
\node[label=right:$P''$] at (0.6,3.9) {};

\end{tikzpicture}
\caption{Case 3: three paths from $w$ to $A'$.}
\label{Th3-case3}
\end{subfigure}

\caption{An illustration of the three cases in the proof of \cref{th:clique-to-Sttt}. Edges and paths highlighted in blue indicate that the their vertices induce a chordless cycle, thereby satisfying Assumption 7 of \cref{lem:clique-3-path}.}
\label{fig:Th3}
\end{figure}
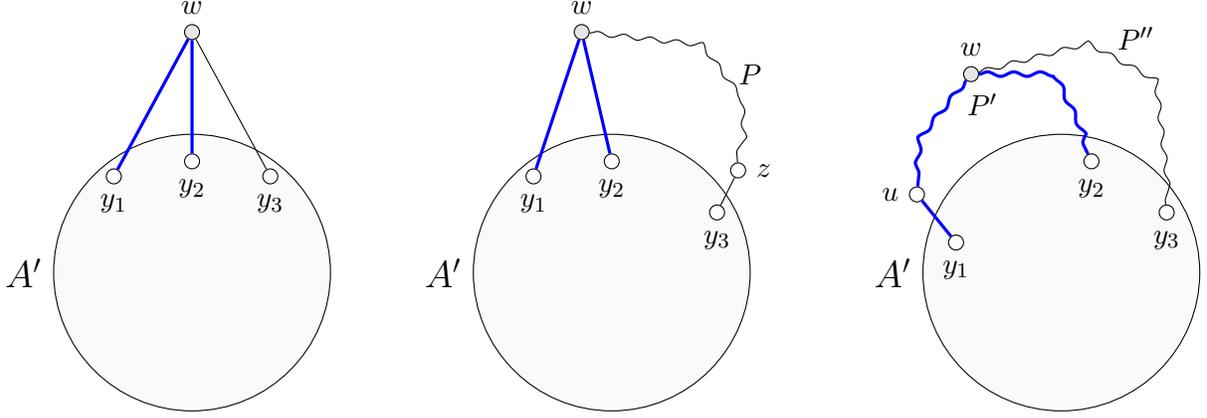
    
    Suppose, towards a contradiction, there exists a connected non-complete graph $G \in \cX_t$ with no clique cutsets of size at most $2$ that contains a clique of size at least $r$.
    Let $A'$ be a maximal clique in $G$ of size at least $r$.
    Since $G$ is non-complete and connected, there exists a vertex in $V(G) \setminus A'$ that has a neighbour in $A'$. 
    We consider three cases (see \cref{fig:Th3}).
    
    \textbf{Case 1:} \textit{There exists a vertex $w \in V(G) \setminus A'$ with at least three neighbours in $A'$.} Since $A'$ is maximal, $w$ also has a non-neighbour in $A'$. Then, by \cref{lem:clique+v-subgraph}, $G \sic H$, where $H$ is a spanning subgraph of $G[A' \cup \{w\}]$ obtained by removing all but three edges incident to $w$. Note that $H$ satisfies the conditions of \cref{lem:clique-3-path}, and thus $G \sic H \sic tS_{t,t,t}$.

    \textbf{Case 2:} \textit{Every vertex outside $A'$ has at most two neighbours in $A'$, and there exists a vertex $w \in V(G) \setminus A'$ with exactly two neighbours in $A'$}.
    Let $y_1,y_2$ be the neighbours of $w$ in $A'$. Since $G$ has no clique cutsets of size $2$, the graph $G-\{y_1,y_2\}$ is connected. Let $P$ be a shortest path from $w$ to $A' \setminus \{y_1,y_2\}$ in this graph, and let $y_3$ denote the endpoint of $P$ in $A' \setminus \{y_1,y_2\}$ and let $z$ denote the vertex on $P$ adjacent to $y_3$. 
    If $z$ has a second neighbour $y_3'$ in $A' \setminus \{y_1,y_2\}$, then we define $H = G[(A' \cup V(P)) \setminus \{y_3'\}]$; otherwise, we define $H = G[A' \cup V(P)]$. We observe that $H$ satisfies the conditions of \cref{lem:clique-3-path} with $A$ being $A' \setminus \{y_3'\}$ or $A'$ respectively, and $Q_1 = (w,y_1)$, $Q_2 = (w,y_2)$, and $Q_3 = P$. Therefore, $G \sic H \sic tS_{t,t,t}$.

    \textbf{Case 3:} \textit{Every vertex outside $A'$ has at most one neighbour in $A'$}. 
    First, we argue that there exists a path with at least three vertices whose endpoints are in $A'$ and all its internal vertices are outside $A'$. Indeed, let $u$ be a vertex outside $A'$ that has a unique neighbour in $A'$, which we denote by $y_1$. Since $G$ has no cut vertices, the graph $G-y_1$ is connected. 
    By taking a shortest path from $u$ to $A'\setminus \{y_1\}$ and extending it with the vertex $y_1$, we obtain a path $P'$ with the desired properties. 
    Without loss of generality, assume that among all such paths, the path $P'$ has the least number of vertices. This, in particular, implies that $G[V(P')]$ is a chordless cycle. 
    Let $y_2 \in A' \setminus \{y_1\}$ denote the second endpoint of $P'$. 
    
    Now, since $G$ has no clique cutsets of size $2$, the graph $G-\{y_1,y_2\}$ is connected. Let $P''$ be a shortest path between $V(P') \setminus \{y_1,y_2\}$ and $A' \setminus \{y_1,y_2\}$ in $G-\{y_1,y_2\}$. Let $w$ be the end vertex of $P''$ in $V(P') \setminus \{y_1,y_2\}$ and $y_3$ be the end vertex of $P''$ in $A' \setminus \{y_1,y_2\}$. 

    Then, the graph $H=G[A' \cup V(P') \cup V(P'')]$ satisfies the conditions of \cref{lem:clique-3-path} with $A$ being $A'$, and $Q_1$, $Q_2$ and $Q_3$ being the subpath of $P'$ from $w$ to $y_1$, the subpath of $P'$ from $w$ to $y_2$ and $P''$, respectively.
    Therefore, $G \sic H \sic tS_{t,t,t}$.
\end{proof}

\subsection{The Structure of Graphs Excluding a Tripod}

In this section, we prove \cref{th:tSttt-main}.
In addition to \cref{th:biclique-to-Sttt} and \cref{th:clique-to-Sttt}, established in the previous sections, we will require two more ingredients. 
One of them is the following implication of the characterization of bounded treewidth in finitely defined hereditary classes.

\begin{theorem}[Lozin and Razgon \cite{LR22}]\label{th:treewidth-dichotomy}
    Let $r \in \bN$, and let $\cX$ be a hereditary graph class that excludes $K_r$, $K_{r,r}$, $rS_{r,r,r}$, and $L(rS_{r,r,r})$ as induced subgraphs. Then there exists a constant $c \in \mathbb{N}$ such that $\tw(G) \leq c$ for every $G \in \cX$.
\end{theorem}

The remaining ingredient is the following observation that the line graph of $tS_{q,q,q}$ can be self-intersected to $tS_{q-1,q-1,q}$.

\begin{lemma}\label[lemma]{lem:LtSttt-sic-tSttt}
    Let $t \geq 1$ and $q \geq 2$. 
    Then $L(tS_{q,q,q})\sic tS_{q-1,q-1,q}.$
\end{lemma}
\begin{proof}
    We give a proof of the statement for $t=1$. For larger $t$ the proof is extended in a straightforward way and we omit the details. 
    See \cref{fig:LSttt-sic-St-1t-1t} for an illustration of the proof.

    Let $H$ be the graph isomorphic to $L(S_{q,q,q})$, with vertex set $[3q]$ and the following edges:
    \begin{enumerate}
        \item $(i,i+1)$ for  $i \in [2q-1]$;
         \item $(i,i+1)$ for  $i \in \{2q+1, 2q+2, \ldots, 3q-1 \}$;
        \item $(2q+1,q)$, $(2q+1,q+1)$.
    \end{enumerate}

    Let $H'$ be the graph isomorphic to $L(S_{q,q,q})$, with vertex set $([3q]\cup \{0\}) \setminus \{2q\}$
    and the following edges:
    \begin{enumerate}
        \item $(i,i+1)$ for  $i \in \{0,1, \ldots, 2q-2\}$;
         \item $(i,i+1)$ for  $i \in \{2q+1, 2q+2, \ldots, 3q-1 \}$;
        \item $(2q+1,q-1)$, $(2q+1,q)$.
    \end{enumerate}

    It is easy to see that $H \cap H'$ is isomorphic to $S_{q-1,q-1,q}$.
\end{proof}

\begin{figure}
\centering
\begin{tikzpicture}[scale=1]
  \tikzstyle{defaultv}=[circle, draw, inner sep=1pt, minimum size=12pt, font=\scriptsize]
  \tikzstyle{defaultredv}=[circle, draw=red!70, inner sep=1pt, minimum size=12pt, font=\scriptsize]
  \tikzstyle{defaultbluev}=[circle, draw=blue!70, inner sep=1pt, minimum size=12pt, font=\scriptsize]
  \tikzstyle{intersection}=[circle, draw=gray!50!black, fill=gray!20, thick, inner sep=1pt, minimum size=12pt, font=\scriptsize]

\def\q{4}

\node[defaultredv] (v0) at (1, 0) {\scriptsize 0};
\node[intersection] (v1) at (2, 0) {\scriptsize 1};
\node[intersection] (v2) at (3, 0) {\scriptsize 2};
\node[intersection] (v3) at (4, 0) {\scriptsize 3};
\node[intersection] (v4) at (5, 0) {\scriptsize 4};
\node[intersection] (v5) at (6, 0) {\scriptsize 5};
\node[intersection] (v6) at (7, 0) {\scriptsize 6};
\node[intersection] (v7) at (8, 0) {\scriptsize 7};
\node[defaultbluev] (v8) at (9, 0) {\scriptsize 8};

\node[intersection] (v9) at (5, 1.5) {\scriptsize 9};
\node[intersection] (v10) at (5, 2.5) {\scriptsize 10};
\node[intersection] (v11) at (5, 3.5) {\scriptsize 11};
\node[intersection] (v12) at (5, 4.5) {\scriptsize 12};

\draw[blue!70, bend right=15] (v7) to (v8);
\foreach \i in {1,...,6} {
    \pgfmathtruncatemacro{\j}{\i + 1}
    \draw[line width=1.2pt, blue!70, bend right=15] (v\i) to (v\j);
}
\draw[line width=1.2pt, blue!70, bend right=15] (v9) to (v10);
\draw[line width=1.2pt, blue!70, bend right=15] (v10) to (v11);
\draw[line width=1.2pt, blue!70, bend right=15] (v11) to (v12);
\draw[line width=1.2pt, blue!70, bend left=15] (v9) to (v4);  
\draw[blue!70] (v9) to (v5);  

\draw[red!70, bend left=15] (v0) to (v1);
\foreach \i in {1,...,6} {
    \pgfmathtruncatemacro{\j}{\i + 1}
    \draw[line width=1.2pt, red!70, bend left=15] (v\i) to (v\j);
}
\draw[line width=1.2pt, red!70, bend left=15] (v9) to (v10);
\draw[line width=1.2pt, red!70, bend left=15] (v10) to (v11);
\draw[line width=1.2pt, red!70, bend left=15] (v11) to (v12);
\draw[red!70] (v9) to (v3);  
\draw[line width=1.2pt, red!70, bend right=15] (v9) to (v4);  

\end{tikzpicture}
\caption{An illustration of the proof of \cref{lem:LtSttt-sic-tSttt}. 
Both graphs $H$ (blue) and $H'$ (red) are isomorphic to $L(S_{4,4,4})$. 
Their intersection (highlighted vertices and thick double edges) is isomorphic to $S_{3,3,4}$.}

\label{fig:LSttt-sic-St-1t-1t}
\end{figure}
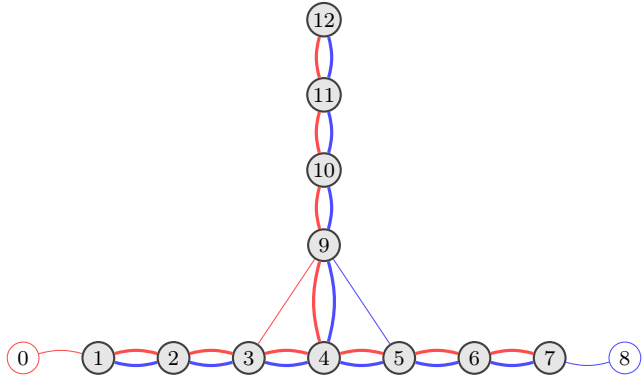

With these tools at hand we can now prove \cref{th:tSttt-main}, which we restate for convenience. 

\tStttmain*
\begin{proof}    
    Let $c \in \bN$ be a constant given by \cref{th:treewidth-dichotomy} for $r=3t^2+t+2$.
    Let $G$ be a graph in $\cX_t$ with at least two vertices, and assume that $G$ contains neither false twins nor clique cutsets of size at most $2$. If $G$ is complete, the claim holds trivially. Thus, assume that $G$ is not complete.  
    We will show that $\tw(G) \leq c$.

    First, by \cref{th:clique-to-Sttt}, graph $G$ is $K_r$-free. 
    Second, we claim that $G$ is $K_{r,r}$-free. Indeed, by \cref{th:biclique-to-Sttt}, any maximal induced complete bipartite graph in $G$ with parts of size $p,q \geq r$ would be a module in $G$, and thus any two vertices in the same part of such a complete bipartite graph would be false twins in $G$, which is not possible by our assumption.
    Finally, due to \cref{lem:LtSttt-sic-tSttt}, $G$ contains no $L(tS_{t+1,t+1,t})$, and in particular no $L(rS_{r,r,r})$, as an induced subgraph.

    Consequently, $G$ contains no $K_r$, $K_{r,r}$, $L(rS_{r,r,r})$, or $rS_{r,r,r}$ as induced subgraphs, and therefore $\tw(G) \leq c$, by \cref{th:treewidth-dichotomy}.
\end{proof}

\section{Algorithmic and Structural Implications}\label{sec:dichotomies}

In this section we provide several applications of our main structural result. In particular, we establish a number of dichotomies in finitely-defined self-intersection-closed graph classes.

To prove these dichotomies, we make use of the following hereditary graph classes.
Recall that $\cS$ denotes the class of graphs every connected component of which is a tripod.
It is known and easy to see that $\cS$ is the limit (i.e.~the intersection) of the infinite sequence of classes $\cS_k \supset \cS_{k+1}\supset \cS_{k+2}\supset\nobreak \cdots$, where $\cS_k=\SubgraphFree(C_3,\ldots,C_k,H_0,H_1,\ldots,H_k)$ for $k\ge 3$, and $H_k$ denotes the graph shown in \cref{fig:Hk}, where $H_0=K_{1,4}$. 
Note that since $\cS_k$ is monotone, it is also self-intersection-closed. 

\begin{observation}\label{obs:Sk}
For every integer $k \geq 3$, there exists an integer $k' \geq k$ such that $\mathcal{S}_{k}$ contains $\Free(C_3,\ldots,C_{k'},H_0,H_1,\ldots,H_{k'})$.
\end{observation}

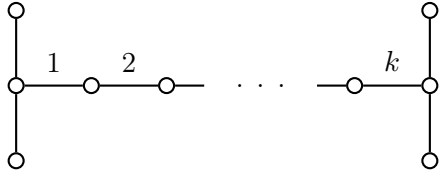
\begin{figure}[ht]
    \centering
    \begin{tikzpicture}[scale=1, every node/.style={circle, draw, fill=white, inner sep=2pt}, thick]
        \centering
        \node (L1) at (0,2) {};
        \node (L2) at (0,1) {};
        \node (L3) at (0,0) {};
        \draw (L1) -- (L2) -- (L3);
        
        \node (R1) at (5.5,2) {};
        \node (R2) at (5.5,1) {};
        \node (R3) at (5.5,0) {};
        \draw (R1) -- (R2) -- (R3);
        
        \node (M1) at (1,1) {};
        \node (M2) at (2,1) {};
        \node[draw=none, fill=none, inner sep=0pt] at (3.25,1) {. . .};
        \node (Mk) at (4.5,1) {};
        
        \draw (L2) -- (M1) node[draw=none, fill=none, inner sep=0pt, midway, above=4pt] {1};
        \draw (M1) -- (M2) node[draw=none, fill=none, inner sep=0pt, midway, above=4pt] {2};
        \draw (Mk) -- (R2) node[draw=none, fill=none, inner sep=0pt, midway, above=4pt] {$k$};
        \draw (M2) -- (2.5,1);
        \draw (4,1) -- (Mk);
    \end{tikzpicture}
    \caption{The graph $H_k$.}
    \label{fig:Hk}
\end{figure}

The importance of the classes $\cS_k$ lies in the fact that various graph parameters (e.g.~treewidth and clique-width) are unbounded and many algorithmic graph problems (e.g.~\mis and \mim) are \textup{\textsf{NP}}-hard on these classes (see, e.g.~\cite{MR4867975}). 
In the following lemma, we show that if a finitely-defined, self-intersection-closed class does not exclude any tripod, then it contains one of the classes $\cS_k$, and thus inherits their structural and algorithmic complexity.

\begin{lemma}\label[lemma]{lem:S_k-inclusion}
    Let $M$ be a finite set of graphs none of which is a tripod, i.e.~$M \cap \cS = \emptyset$. Then there exists $r \in \bN$ such that $\cS_r \subseteq \SiFree(M)$. 
\end{lemma}
\begin{proof}
    Since $M \cap \cS = \emptyset$, $\cS = \bigcap_{k \geq 3} \cS_{k}$, and $\cS_3 \supset \cS_4 \supset \cS_5 \supset \cdots$, each graph in $M$ belongs to at most finitely many classes $\cS_{k}$ (possibly to none of them). Therefore, since $M$ is finite, collectively graphs in $M$ belong to at most finitely many classes $\cS_{k}$, and hence there is a finite $r$ such that $M \cap \cS_r=\emptyset$.
    As a result, $\cS_r \subseteq \SiFree(M)$,
    since $\cS_r$ is self-intersection-closed.
\end{proof}

We now use our structural results about the classes $\cX_t = \SiFree(tS_{t,t,t})$, $t \geq 1$, to show that if $M$ contains at least one tripod, the class $\SiFree(M)$ becomes simple with respect to various structural parameters and algorithmic problems.

\subsection{Dichotomy for the \mmaybewis Problem}

\begin{theorem}\label{thm:MIS}
    Let $M$ be a finite set of graphs and let $\cY=\SiFree(M)$. 
    Then, \textup{\mis} and \textup{\mwis} are polynomial-time solvable on the class $\cY$ if $M\cap \cS \ne \emptyset$, and \textup{\textsf{NP}}-hard otherwise.
\end{theorem}
\begin{proof}
    Suppose first that $M \cap \cS = \emptyset$. 
    Then, by \cref{lem:S_k-inclusion}, $\cY$ contains $\cS_r$ for some $r \geq 3$. 
    This implies that \mis (and therefore \mwis) is \textup{\textsf{NP}}-hard in $\cY$, because it is \textup{\textsf{NP}}-hard in $\cS_r$ for all $r \geq 3$ as a consequence of \cite[Proposition~2]{MR3695266} and \Cref{obs:Sk}.
    
    Now suppose that $M$ contains a graph $H$ from $\cal S$. Without loss of generality, we may assume that $H=tS_{t,t,t}$ for some $t \in \bN$. Let $G$ be an arbitrary graph in $\SiFree(tS_{t,t,t})$. 
    We may assume that $G$ is not complete, as otherwise \mwis is trivial on $G$.
    Combining modular decomposition and decomposition by clique separators as in \cite{MR4586793}, we may further assume that $G$ has no modules other than $V(G)$ and $\{v\}$, $v\in V(G)$, and has no clique cutsets. 
    Then, from \cref{th:biclique-to-Sttt,th:clique-to-Sttt}, we conclude that $G$ contains neither large bicliques nor large cliques as induced subgraphs. 
    It also does not contain the line graph of  $tS_{t+1,t+1,t+1}$, because this graph can be self-intersected to $tS_{t,t,t}$ by \cref{lem:LtSttt-sic-tSttt}. 
    Therefore, due to \cref{th:treewidth-dichotomy}, the treewidth of $G$ is bounded. 
    Consequently, by~\cite{MR1105479,MR1417901}, the \mwis problem can be solved in polynomial time on graphs from~$\cY$. 
\end{proof}

In the rest of this section, we use $\cB$ to denote the class of bipartite graphs.

\subsection{Clique-width Dichotomy on Bipartite Graphs}

\begin{theorem}\label{th:bip-cw-dichotomy}
    Let $M$ be a finite set of bipartite graphs and let $\cY = \SiFree(M) \cap \cB$. Then $\cY$ has bounded clique-width if and only if $M\cap \cS \ne \emptyset$.
\end{theorem}
\begin{proof}
    Suppose that $M \cap \cS = \emptyset$. Then, by \cref{lem:S_k-inclusion}, $\cY$ contains $\cS_k\cap \cB$ for some $k \geq 3$. 
    This implies that $\cY$ has unbounded clique-width, because bipartite graphs in $\cS_k$ have unbounded clique-width for all $k \geq 3$ 
    (see, e.g.~\cite{MR2195310}).

    Now suppose that $M \cap \cS \ne \emptyset$.
    Then, for some constant $t \in \bN$, it follows that $\cY \subseteq \cX_t \cap \cB$.
    By \Cref{lem:graph-sic-induced-subgraph}, graphs in $\cY$ do not contain $tS_{t,t,t}$ as an induced subgraph. 
    Furthermore, since graphs in $\cY$ are bipartite, they do not contain $K_3$, and therefore do not contain $L(S_{1,1,1})$, the line graph of $S_{1,1,1}$. 
    Now, if a biconnected graph $G$ in $\cY$ contains an induced biclique $K_{p,q}$ with $p,q \geq 3t^2 + t + 1$, then, by \cref{th:biclique-to-Sttt}, the vertex set of any maximal extension $H$ of this biclique is a module in $G$.
    This implies that $G= H$ is a biclique, since otherwise any vertex in $V(G)\setminus V(H)$ that has a neighbour in $H$ is complete to $V(H)$ and would hence lead to a $K_3$ in $G$, contradicting the fact that $G$ is bipartite. 
    Thus, any biconnected graph in $\cY$ is either a biclique or belongs to $\Free(tS_{t,t,t}, L(S_{1,1,1}), K_3, K_{p,p})$, where $p = 3t^2 + t + 1$. 
    Since the latter class has bounded treewidth (by \cref{th:treewidth-dichotomy}), and bicliques have clique-width $2$, we conclude that the class of biconnected graphs in $\cY$ has bounded clique-width. 
    The result now follows from the fact that, in any hereditary graph class, boundedness of clique-width in the subclass of biconnected graphs implies boundedness of clique-width in the entire class~\cite[Proposition 2]{MR2112498}.
\end{proof}

\subsection{Dichotomy for \maybecsat}

The \emph{incidence graph of a CNF formula} $\varphi$ is a bipartite graph $G_\varphi$, one part of which represents clauses and the other represents variables. The edges of $G_\varphi$ connect variables to the clauses containing them (positively or negatively). 

\begin{theorem}\label{thm:SAT}
    Let $M$ be a finite set of bipartite graphs and let $\cY = \SiFree(M) \cap \cB$. 
    \begin{enumerate}[(1)]
        \item\label{satNP-c} If $M\cap {\cal S} = \emptyset$, then the \textup{\sat} problem is \textup{\textup{\textsf{NP}}}-complete on instances with incidence graphs in $\cY$.
        
        \item If $M\cap {\cal S} \ne \emptyset$, then the \textup{\csat} problem is polynomial-time solvable on instances with incidence graphs in $\cY$.
    \end{enumerate}
\end{theorem}
\begin{proof}
    Suppose that $M \cap \cS = \emptyset$. Then, by \cref{lem:S_k-inclusion}, $\cY$ contains $\cS_k \cap \cB$ for some $k \geq 3$. This immediately implies (1), as the \sat problem is \textup{\textsf{NP}}-complete on instances with incidence graphs in the class $\cS_k \cap \cB$ for any $k\ge 3$, as a consequence of \cite[Lemma~3]{MR3033644} and \Cref{obs:Sk}. 

    Now suppose that $M \cap \cS \ne \emptyset$. Then, by \cref{th:bip-cw-dichotomy}, the class $\cY$ has bounded clique-width, and therefore, according to \cite{SS13}, the \csat problem can be solved in polynomial time on instances with incidence graphs in $\cY$. 
\end{proof}

\subsection{Dichotomy for the \mim problem on Bipartite Graphs}

\begin{theorem}\label{th:mim-dichotomy}
    Let $M$ be a finite set of bipartite graphs and let $\cY=\SiFree(M) \cap \cB$. 
    Then the \textup{\mim} problem is polynomial-time solvable on the class $\cY$ if $M\cap \cS \ne \emptyset$, and \textup{\textsf{NP}}-hard otherwise.
\end{theorem}
\begin{proof}
    First suppose that $M \cap \cS = \emptyset$. Then, by \cref{lem:S_k-inclusion}, for some $k \geq 3$, the class $\cY$ contains $\cS_k \cap \cB$. 
    This implies that the \mim problem is \textup{\textsf{NP}}-hard on $\cY$, because it is \textup{\textsf{NP}}-hard on $\cS_k \cap \cB$ for all $k \geq 3$ as a consequence of~\cite[Corollary~1]{Loz02} and \Cref{obs:Sk}.

    Now suppose that $M \cap \cS \ne \emptyset$. Then, by \cref{th:bip-cw-dichotomy}, the class $\cY$ has bounded clique-width, and therefore the \mim problem can be solved in polynomial time on $\cY$ (see, e.g.~\cite{MR1739644,MR2232389}). 
\end{proof}

\iftoggle{anonymous}{
}{%
	\bigskip
	\textbf{Acknowledgments.} This work was initiated while the authors were staying at the birthplace of James Clerk Maxwell in Edinburgh for a Research-in-Groups visit supported by the International Centre for Mathematical Sciences.
}

\bibliographystyle{abbrv}

\end{document}